\documentclass[12pt]{amsart}

  \usepackage{graphicx}
  \usepackage{pictexwd,dcpic}
  \usepackage{latexsym,amscd,amssymb,epsfig,a4wide}

  \def\RR{{\mathbb R}}

  \def\sm{\smallsetminus}

  \def\fvec{{\mathfrak f}}
  \def\gvec{{\mathfrak g}}
  
  \def\ppgvec{\mathop{\mathfrak g}''}
  
  \def\hvec{{\mathfrak h}}
  \def\phvec{{\mathfrak h}'}
  \def\pphvec{\mathop{\mathfrak h}''}
  \def\pphvec{{\mathfrak h}''}

  \def\gG{{\mathcal G}}

  \def\link{{\mathrm{lk}}}

  \def\cost{{\rm cost}}

  \def\wbeta{{\widetilde{\beta}}}

  \theoremstyle{plain}
    \newtheorem{theorem}{Theorem}[section]
    \newtheorem{proposition}[theorem]{Proposition}
    
    \newtheorem{corollary}[theorem]{Corollary}
    
  \theoremstyle{definition}
    \newtheorem{definition}[theorem]{Definition}
    \newtheorem{example}[theorem]{Example}
    
    \newtheorem{question}[theorem]{Question}
    \newtheorem{remark}[theorem]{Remark}

  \numberwithin{equation}{section}

\begin{document}

  \title[Buchsbaum* Complexes]{Buchsbaum* Complexes}

  \author{Christos~A.~Athanasiadis}
  \address{Department of Mathematics \\
           Division of Algebra-Geometry \\
           University of Athens \\
           Panepistimioupolis, Athens 15784 \\
           Greece}
  \email{caath@math.uoa.gr}

  \author{Volkmar~Welker}
  \address{Fachbereich Mathematik und Informatik\\
           Philipps-Universit\"at Marburg\\
           35032 Marburg, Germany}
  \email{welker@mathematik.uni-marburg.de}

  \date{November 9, 2010}

  \begin{abstract}
    A class of finite simplicial complexes, which we call Buchsbaum* over a field, is 
    introduced. Buchsbaum* complexes generalize triangulations of orientable homology 
    manifolds as well as doubly Cohen-Macaulay complexes. By definition, the Buchsbaum* 
    property depends only on the geometric realization and the field. Characterizations
    in terms of simplicial homology are given. It is proved that 
    Buchsbaum* complexes are doubly Buchsbaum. Various constructions, among them one 
    which generalizes convex ear decompositions, are shown to yield Buchsbaum* simplicial 
    complexes. Graph theoretic and enumerative properties of Buchsbaum* complexes are 
    investigated.
  \end{abstract}

  \maketitle

  \section{Introduction}
  \label{sec:intro}

  A major theme in the study of finite simplicial complexes in the past few decades has 
  been the interplay between their algebraic, combinatorial, homological and topological 
  properties. Several classes of simplicial complexes, such as Buchsbaum, Cohen-Macaulay 
  or Gorenstein complexes, have been introduced and studied in order to isolate important 
  features of triangulations of fundamental geometric objects, such as balls, spheres and 
  various other manifolds. We refer the reader to \cite{Sta} for a comprehensive 
  introduction to the subject. The objective of this paper is to introduce and develop 
  the basic properties of a new class of simplicial complexes, named Buchsbaum* complexes, 
  which generalize triangulations of orientable homology manifolds.

  In this introductory section we motivate our main definition and outline the remainder 
  of the paper (we refer the reader to \cite[Chapter II]{Sta} \cite[Chapter 5]{BH} 
  \cite[Chapter II]{SV} for any undefined terminology and to \cite{Mu} for background on
  algebraic topology). A \emph{simplicial complex} over the ground set $\Omega$ is a 
  collection $\Delta$ of subsets of $\Omega$ such that $\sigma \subseteq \tau \in \Delta$ 
  implies $\sigma \in \Delta$. Throughout this paper, we will always assume $\Omega$ to be
  finite (thus we will only consider finite simplicial complexes). Given a field $k$, the 
  \emph{face ring} (or \emph{Stanley-Reisner ring}) $k[\Delta]$ of $\Delta$ over $k$ is the 
  quotient of the polynomial ring $k[x_\omega: \omega \in \Omega]$ by the ideal generated 
  by the monomials $\prod_{\omega \in N} x_\omega$ for all subsets $N$ of $\Omega$ not in 
  $\Delta$. Recall (see \cite{Sta}) that $\Delta$ is called \emph{Buchsbaum} (respectively, 
  \emph{Cohen-Macaulay}, \emph{Gorenstein}) over $k$ if the face ring $k[\Delta]$ is a 
  Buchsbaum (respectively, Cohen-Macaulay, Gorenstein) ring. Such a complex $\Delta$ is 
  called \emph{doubly Buchsbaum} \cite{Mi} (respectively, \emph{doubly Cohen-Macaulay} 
  \cite{Ba} \cite[p.~71]{Sta}) over $k$ if for every vertex $v$ of $\Delta$, the complex 
  $\Delta \sm v$, obtained from $\Delta$ by removing all faces which contain $v$, is 
  Buchsbaum (respectively, Cohen-Macaulay) over $k$ of the same dimension as $\Delta$.
 
  It is known that Buchsbaumness \cite{Sch}\cite[Theorem 8.1]{Sta}; see also Theorem 
  \ref{BuchsbaumCharacterization} (respectively, Cohen-Macaulayness; see \cite{Mu2} 
  \cite[Proposition 4.3]{Sta}) of $\Delta$ is a topological property, meaning that it 
  depends only on the homeomorphism type of the geometric realization $|\Delta|$ 
  \cite[Section 9]{Bj} of $\Delta$. For instance, all triangulations of manifolds (with 
  or without boundary) are Buchsbaum and all triangulations of balls and spheres are 
  Cohen-Macaulay over all fields. It was conjectured by Baclawski \cite{Ba} and proved 
  by Walker \cite{Wa} as an immediate consequence of the following theorem, that double 
  Cohen-Macaulayness is a topological property as well. In the sequel, we will write 
  $\widetilde{H}_i(X;k)$ for the reduced singular homology of the space $X$ and $H_i 
  (X,A;k)$ for the singular homology of the pair of spaces $(X,A)$. We will also write 
  $\widetilde{H}_i(\Delta;k)$ and $H_i (\Delta, \Gamma;k)$ for the (reduced) simplicial 
  homology of the simplicial complex $\Delta$ and of the pair of simplicial complexes 
  $(\Delta, \Gamma)$. 

  \begin{theorem} {\rm (Walker, \cite[Theorem 9.8]{Wa})}  \label{Wa2-CM}
  Let $\Delta$ be a $(d-1)$-dimensional Cohen-Macaulay simplicial complex 
  over a field $k$. The following conditions are equivalent:
     \begin{itemize}
        \item[(i)] $\Delta$ is doubly Cohen-Macaulay over $k$.
        \item[(ii)] $\widetilde{H}_{d-2} (|\Delta|-p;k) = 0$ holds
                    for every $p  \in |\Delta|$. 
     \end{itemize}
  \end{theorem}

  Examples of complexes which are doubly Cohen-Macaulay over all fields are all 
  triangulations of spheres. In contrast, no triangulation of a ball is doubly 
  Cohen-Macaulay over any field. 

  Double Buchsbaumness of a simplicial complex is also a topological property 
  \cite{Mi} and thus doubly Buchsbaum complexes generalize homology manifolds 
  (without boundary) in a way analogous to the way doubly Cohen-Macaulay complexes 
  generalize homology spheres. However, in certain respects double Buchsbaumness 
  turns out to be too weak of an analogue of double Cohen-Macaulayness. For instance, 
  it is known \cite[p.~71]{Sta} that every doubly Cohen-Macaulay complex $\Delta$ 
  has non-vanishing top-dimensional homology, whereas this is not true for every 
  doubly Buchsbaum complex (since it is not true for every homology manifold 
  without boundary). These considerations and condition (ii) in Theorem 
  \ref{Wa2-CM} motivate the following definition.

  \begin{definition} \label{def:Buch*}
    Let $\Delta$ be a $(d-1)$-dimensional Buchsbaum simplicial complex over a 
    field $k$. The complex $\Delta$ is called \emph{Buchsbaum* over $k$} if 
    \begin{equation} \label{eq:Buch*}
      \dim_k \widetilde{H}_{d-2}(|\Delta|-p;k) = 
                  \dim_k \widetilde{H}_{d-2}(|\Delta|; k)  
    \end{equation} 
     holds for every $p  \in |\Delta|$. 
  \end{definition}

  The results of this paper show that the notion of a Buchsbaum* complex provides a 
  well behaved manifold analogue to that of a doubly Cohen-Macaulay complex in terms 
  of various homological, graph theoretic and enumerative properties. We summarize 
  some of these results as follows. 

  The class of Buchsbaum* complexes is shown to be included in the class of doubly 
  Buchsbaum complexes (Corollary \ref{MainTheorem}) with non-vanishing top-dimensional 
  homology (Corollary \ref{NonVanishingTopHomology}), to include all triangulations of 
  orientable homology manifolds (Proposition \ref{Manifold}) and to reduce to the class 
  of doubly Cohen-Macaulay complexes, when restricted to the class of all Cohen-Macaulay 
  complexes (Proposition \ref{CMBuchsbaum}). Products of Buchsbaum* complexes and proper 
  skeleta of Buchsbaum complexes are shown to be Buchsbaum* (Propositions \ref{Product} 
  and \ref{1-coskeleton}). A notion of higher Buchsbaum* connectivity, generalizing that 
  of higher Cohen-Macaulay connectivity \cite{Ba}, is introduced in Section 
  \ref{sec:skeleta}, where it is shown that passing to proper skeleta increases the 
  degree of connectivity (Theorem \ref{skeletaconnectivity}). 

  Partially extending results of Kalai \cite{Ka} on homology manifolds and Nevo \cite{Ne} 
  on doubly Cohen-Macaulay complexes, it is shown that the graph of a connected Buchsbaum* 
  complex of dimension $d-1 \ge 2$ is generically $d$-rigid (Theorem \ref{graph}). This 
  implies that the face numbers of such a complex satisfy the inequalities of Barnette's 
  lower bound theorem (Proposition \ref{prop:LBT}) and part of the conditions predicted 
  by the $g$-conjecture for triangulations of spheres (Proposition \ref{prop:g}). An 
  analogue of a recursive formula for the $\hvec$-vector of a pure simplicial complex, in 
  terms of that of the deletion and the link of a vertex, is shown to be valid for 
  Buchsbaum* complexes (Propositions \ref{Deletion} and \ref{Deletion2}) and an 
  application to the face enumeration of flag Buchsbaum* complexes is given (Corollary 
  \ref{Flag}). 

  This paper is structured as follows. Section \ref{sec:char} gives characterizations of 
  Buchsbaum* complexes, deduces their basic properties and lists examples. Section 
  \ref{sec:examples} investigates the behavior of the Buchsbaum* property under standard 
  operations on topological spaces and constructs a large family of Buchsbaum* complexes 
  by suitably gluing orientable homology manifolds with boundary to an orientable homology 
  manifold without boundary (Corollary \ref{Class}).
  Sections \ref{sec:graph} and \ref{sec:enumeration} focus on graph theoretic and 
  enumerative properties. Section \ref{sec:questions} briefly discusses some further 
  properties of Buchsbaum* complexes which appeared in the literature after the results 
  of this paper were first publicized.

  \section{Characterizations and elementary properties}
  \label{sec:char}

  This section provides characterizations and discusses basic properties of Buchsbaum* 
  complexes. Throughout this paper, if not specified otherwise, $k$ is an arbitrary 
  field. We recall the following characterization of Buchsbaum 
  complexes. 

  \begin{theorem}[Schenzel \cite{Sch}]  \label{BuchsbaumCharacterization}
     For a $(d-1)$-dimensional simplicial complex $\Delta$, the following 
     conditions are equivalent:
     \begin{itemize}
        \item[(i)] $\Delta$ is Buchsbaum over $k$.
        
       \item[(ii)] $\Delta$ is pure and $\link_\Delta(\sigma)$ is Cohen-Macaulay 
                    over $k$ for every $\sigma \in \Delta \sm \{ \varnothing \}$. 
       \item[(iii)] $H_i (|\Delta|, |\Delta|-p;k) = 0$ holds for 
                    all $i < d-1$ and $p \in |\Delta|$.
     \end{itemize}
   \end{theorem}

    \begin{remark} \label{rem:components}
       {\rm Let $\Delta$ be a $(d-1)$-dimensional simplicial complex. The following 
   statements are immediate consequences of Theorem 2.1 and Definition 1.2.
       \begin{itemize}
          \item[(i)] $\Delta$ is Buchsbaum over $k$ if and only
            if every connected component of $\Delta$ is Buchsbaum
            over $k$ of dimension $d-1$.
          \item[(ii)] Assume that $d \ge 2$. Then $\Delta$ is
            Buchsbaum* over $k$ if and only if every connected
            component of $\Delta$ is Buchsbaum* over $k$ of
            dimension $d-1$.
       \end{itemize}}
   \end{remark}

   The following proposition provides equivalent versions of Definition \ref{def:Buch*}.

  \begin{proposition} \label{Isomorphism} 
     For a $(d-1)$-dimensional simplicial complex $\Delta$ which is Buchsbaum over 
     $k$, the following conditions are equivalent:
     \begin{itemize}
        \item[(i)]  $\Delta$ is Buchsbaum* over $k$.
       \item[(ii)]  For every $p \in |\Delta|$, the inclusion map $\iota : 
                    |\Delta| - p \hookrightarrow |\Delta|$ induces an injection
                    $$\iota_* : \widetilde{H}_{d-2}(|\Delta| - p;k) \rightarrow  
                    \widetilde{H}_{d-2}(|\Delta| ; k).$$
      \item[(iii)]  For every $p \in |\Delta|$, the inclusion map $\iota : 
                    |\Delta| - p \hookrightarrow |\Delta|$ induces an isomorphism 
                    $$\iota_* : \widetilde{H}_{d-2}(|\Delta| - p;k) \rightarrow  
                    \widetilde{H}_{d-2}(|\Delta| ; k).$$
        \item[(iv)] For every $p \in |\Delta|$, the canonical map
                    $$\rho_* : \widetilde{H}_{d-1}(|\Delta|;k) \rightarrow 
                    H_{d-1}(|\Delta|,|\Delta|-p;k)$$
                    is surjective.
     \end{itemize}
  \end{proposition}
  \begin{proof}
     Since $\Delta$ is Buchsbaum over $k$, we have $H_{d-2} (|\Delta|, 
     |\Delta|-p;k) = 0$ by condition (iii) of Theorem \ref{BuchsbaumCharacterization}. 
     Hence, the long exact sequence of the pair $(|\Delta|,|\Delta|-p)$ gives the exact 
     sequence

     \medskip
     
     $$
       \begindc{\commdiag}[3]
          \obj(0,50){$0$}
          \obj(25,50){$\widetilde{H}_{d-1} (|\Delta|-p;k)$}
          \obj(60,50){$\widetilde{H}_{d-1}(|\Delta|;k)$}
          \obj(98,50){$H_{d-1}(|\Delta|,|\Delta|-p;k)$}
          \mor(0,50)(13,50){$~$}
          \mor(38,50)(52,50){}
          \mor(69,50)(83,50){$\rho_*$}
          \obj(25,35){$\widetilde{H}_{d-2}(|\Delta|-p;k)$}
          \obj(60,35){$\widetilde{H}_{d-2}(|\Delta|;k)$}
          \obj(83,35){$0$.}
          \mor(38,35)(52,35){$\iota_*$}
          \mor(69,35)(83,35){$~$}
          \cmor((100,45)(100,43)(100,41)(97,41)(95,41)(90,41)(75,41)(55,41)(40,41)(35,41)(34,41)(33,41)(30,41)(30,39)(30,37)) 
          \pdown(60,43){$$}
       \enddc
     $$

     \medskip

     \noindent 
     It follows that $\iota_*$ is surjective. This proves that ${\rm (ii)} 
     \Leftrightarrow {\rm (iii)}$. Assuming that $\Delta$ is Buchsbaum* over $k$,  
     surjectivity of $\iota_*$ and (\ref{eq:Buch*}) imply that $\iota_*$ is an 
     isomorphism. This proves that ${\rm (i)} \Rightarrow {\rm (iii)}$. The reverse 
     implication is trivial. The same exact sequence proves the equivalence ${\rm (iii)} 
     \Leftrightarrow {\rm (iv)}$.  
  \end{proof}

  \begin{corollary} \label{NonVanishingTopHomology}
     If $\Delta$ is a $(d-1)$-dimensional Buchsbaum* simplicial complex over $k$, 
     then $$\widetilde{H}_{d-1}(\Delta;k) \ne 0.$$ 
  \end{corollary}
  \begin{proof}
     Let us choose $p \in |\Delta|$ in the relative interior of a $(d-1)$-dimensional 
     face of $\Delta$. Clearly, we have $H_{d-1}(|\Delta|, |\Delta|-p;k) 
     \cong k$. The desired statement follows by applying condition (iv) of Proposition 
     \ref{Isomorphism} to such a point $p$.  
  \end{proof}

  We note that part (i) of the next proposition fails if Buchsbaum* is replaced by 
  doubly Buchsbaum (see, for instance, part (i) of Example \ref{ex:counter}).

 \begin{proposition} \label{CMBuchsbaum}
     Let $\Delta$ be a simplicial complex.
     \begin{itemize}
        \item[(i)] Assume that $\Delta$ is Cohen-Macaulay over $k$. 
             Then $\Delta$ is Buchsbaum* over $k$ if and only if $\Delta$ is doubly 
             Cohen-Macaulay over $k$. 
        \item[(ii)] Assume that $\Delta$ is Gorenstein over $k$. 
             Then $\Delta$ is Buchsbaum* over $k$ if and only if $\Delta$ is Gorenstein*
             over $k$. 
     \end{itemize}
  \end{proposition}
  \begin{proof}
    The assumption that $\Delta$ is Cohen-Macaulay over $k$ implies that $\Delta$ is 
    Buchsbaum over $k$ and that $\widetilde{H}_{d-2} (\Delta;k)=0$, where $d-1$ is the 
    dimension of $\Delta$. Therefore, under this assumption, Definition \ref{def:Buch*} 
    implies that $\Delta$ is Buchsbaum* over $k$ if and only if we have 
    $\widetilde{H}_{d-2} (\Delta-p;k)=0$ for every $p \in |\Delta|$. Thus, part (i) 
    follows from Theorem \ref{Wa2-CM}.

    Assume that $\Delta$ is Gorenstein over $k$. This assumption also implies that 
    $\Delta$ is Buchsbaum over $k$. A Gorenstein simplicial complex $\Gamma$ of 
    dimension $d-1$ is Gorenstein* if and only if $\widetilde{H}_{d-1}(\Gamma;k) \ne 0$. 
    Thus if $\Delta$ is Buchsbaum* over $k$, then $\Delta$ is Gorenstein* 
    over $k$ by Corollary \ref{NonVanishingTopHomology}. Conversely, if $\Delta$ is 
    Gorenstein* over $k$, then $\Delta$ is doubly Cohen-Macaulay over $k$ and hence 
    it is Buchsbaum* over $k$ by part (i). This proves part (ii).
  \end{proof}

  \begin{example} \label{ex:1dim}
    {\rm A zero-dimensional simplicial complex is Buchsbaum* over $k$ if and only if 
    it has at least two vertices. Suppose $\Delta$ is one-dimensional, so that $\Delta$ 
    is a graph. Then by Remark \ref{rem:components} (ii), $\Delta$ is Buchsbaum* over $k$ 
    if and only if so is each connected component of $\Delta$. Since a graph regarded as 
    a one-dimensional simplicial complex is Cohen-Macaulay over $k$ if and only if it is 
    connected, we conclude from Proposition \ref{CMBuchsbaum} (i) that $\Delta$ is 
    Buchsbaum* over $k$ if and only if each connected component of $\Delta$ is doubly 
    connected as a graph.
    \qed}
  \end{example}

  By the term \emph{homology manifold} (without further specification) in this 
  paper, we will always mean one without boundary.

  \begin{proposition} \label{Manifold}
    Let $\Delta$ be a triangulation of a homology manifold $X$ over $k$. Then 
    $\Delta$ is Buchsbaum* over $k$ if and only if $X$ is orientable over $k$.
  \end{proposition}
  \begin{proof}
    In view of Remark \ref{rem:components} (ii), we may assume that $|\Delta|$ is 
    connected. Let $d-1$ be the dimension of $\Delta$ and let $p \in |\Delta|$. Our 
    assumptions imply that $\Delta$ is Buchsbaum over $k$, that $\widetilde{H}_{d-1}
    (|\Delta|-p;k) = 0$ and 
    that $H_{d-1} (|\Delta|,|\Delta|-p;k) \cong k$. Thus, the long 
    exact homology sequence considered in the proof of Proposition \ref{Isomorphism} 
    shows that the canonical map $\rho_* : \widetilde{H}_{d-1}(|\Delta|;k) \rightarrow 
    H_{d-1}(|\Delta|,|\Delta|-p;k)$ is surjective if and only if 
    $\widetilde{H}_{d-1} (|\Delta|;k) \cong k$. Since the latter holds if and only if 
    $X$ is orientable over $k$, the proof follows from the equivalence  ${\rm (i)} 
    \Leftrightarrow {\rm (iv)}$ in Proposition \ref{Isomorphism}. 
  \end{proof} 

  The following proposition provides another equivalent version of Definition 
  \ref{def:Buch*}. The condition in part (ii) of this proposition is a stronger version 
  of one which appeared in \cite{Mi} (see also the proof of Corollary \ref{MainTheorem} 
  below). Recall that the \emph{contrastar} of a face $\sigma$ of a simplicial complex 
  $\Delta$ is defined as the subcomplex $\cost_\Delta(\sigma) = \{\tau \in \Delta: \sigma 
  \not \subseteq \tau\}$ of $\Delta$. 

  \begin{proposition} \label{Local}
     For a $(d-1)$-dimensional simplicial complex $\Delta$ which is Buchsbaum over
     $k$, the following conditions are equivalent:
     \begin{itemize}
        \item[(i)]  $\Delta$ is Buchsbaum* over $k$.
        
       \item[(ii)]  For every pair of faces $\sigma \subseteq \tau$ of $\Delta$,
          the map 
          \begin{equation} \label{costarmaps} 
            j_* : H_{d-1} (\Delta, \cost_\Delta(\sigma);k) 
            \rightarrow H_{d-1} (\Delta, \cost_\Delta(\tau);k),    
          \end{equation}      
          induced by inclusion, is surjective.
     \end{itemize}
  \end{proposition}
  \begin{proof}
    Recall that for $p \in |\Delta|$ there is a deformation retraction of $|\Delta|-p$
    onto $|\cost_\Delta(\tau)|$, where $\tau$ is the unique face of $\Delta$ such that 
    $p$ lies in the relative interior of $|\tau|$. As a result, condition (iv) of 
    Proposition \ref{Isomorphism} is equivalent to the condition that for each $\tau \in 
    \Delta$, the canonical map
    $$\rho_*^\tau : \widetilde{H}_{d-1}(\Delta;k) \rightarrow H_{d-1} (\Delta, 
    \cost_\Delta(\tau);k)$$
    is surjective. The commutative diagram of canonical maps
    
    \medskip

     $$
       \begindc{\commdiag}[3]
        \obj(30,30){$\widetilde{H}_{d-1}(\Delta;k)$}
        \obj(85,10){$H_{d-1} (\Delta, \cost_\Delta(\tau);k)$} 
        \obj(85,50){$H_{d-1} (\Delta, \cost_\Delta(\sigma);k)$} 
        \mor(38,28)(68,10){$\rho_*^\tau$}
        \mor(38,32)(68,50){$\rho_*^\sigma$}
        \mor(85,48)(85,12){$j_*$}
      \enddc
    $$

    \medskip

    \noindent 
    for pairs $\sigma \subseteq \tau$ of faces of $\Delta$ shows that the latter condition 
    is equivalent to (ii). 
  \end{proof}

  \begin{corollary} \label{MainTheorem}
     Every Buchsbaum* complex over $k$ is doubly Buchsbaum over $k$. 
  \end{corollary}
  \begin{proof}
     This statement follows from the implication ${\rm (i)} \Rightarrow {\rm (ii)}$
     of Proposition \ref{Local} and the fact (see \cite[Theorem 4.3]{Mi}) that a 
     $(d-1)$-dimensional simplicial complex $\Delta$ is doubly Buchsbaum over $k$ if 
     and only if $\Delta$ is Buchsbaum over $k$ and the map (\ref{costarmaps}) is 
     surjective for every pair of nonempty faces $\sigma \subseteq \tau$ of $\Delta$. 
  \end{proof}

  The following property of Buchsbaum* complexes was proved for connected orientable 
  homology manifolds in \cite[Theorem 2.1]{NS}. 
  \begin{corollary} \label{cor:local}
    Let $\Delta$ be a $(d-1)$-dimensional Buchsbaum* simplicial complex over $k$. 
    Then the socle of the local cohomology module of $k[\Delta]$ in homological 
    dimension $d$ with respect to the irrelevant ideal satisfies $\Big( {\rm Soc} \, 
    H^d (k[\Delta])\Big)_i = 0$ for all $i \ne 0$.
  \end{corollary}
  \begin{proof}
    As noted in the proof of \cite[Theorem 2.1]{NS}, it follows from \cite[Theorem 
    2]{Gra} that the conclusion of the corollary holds if the map 
      $$ j^* : H^{d-1} (\Delta, \cost_\Delta(\tau);k) 
            \rightarrow H^{d-1} (\Delta, \cost_\Delta(\tau \sm \{l\});k) $$  
    in simplicial cohomology, induced by the identity map, is injective for every 
    $\tau \in \Delta \sm \{ \varnothing \}$ and $l \in \tau$. This holds if 
    $\Delta$ is Buchsbaum* over $k$ by Proposition \ref{Local}.
  \end{proof}

  \begin{example} \label{ex:counter}
    {\rm Some examples of doubly Buchsbaum complexes which are not Buchsbaum* are 
    the following.
    
    \begin{itemize} 
      \item[(i)] The one-dimensional simplicial complex $\Delta$ on the vertex set $\{a, b, 
        c, d, p\}$ with facets (edges) $\{p, a\}$, $\{p, b\}$, $\{a, b\}$, $\{p, c\}$, 
        $\{p, d\}$, $\{c, d\}$ is doubly Buchsbaum but not Buchsbaum*, since $\widetilde{H}_0 
        (|\Delta|;k) = 0$ and $\widetilde{H}_0 (|\Delta| - p;k) \cong k$ (alternatively,
        since $\Delta$ is not doubly connected as a graph).

    \item[(ii)] Condition (ii) of Theorem \ref{BuchsbaumCharacterization} and the fact that
        all homology spheres over $k$ are doubly Cohen-Macaulay over $k$ imply that 
        all homology manifolds over $k$ are doubly Buchsbaum over $k$. This fact and 
        Proposition \ref{Manifold} imply that every non-orientable homology manifold over 
        $k$ is doubly Buchsbaum but not Buchsbaum* over $k$. 

    \item[(iii)] Let $\Gamma$ be a triangulation of the two-dimensional torus for which
        some three edges of $\Gamma$ of the form $\{a, b\}$, $\{b, c\}$ and $\{c, a\}$ 
        are the support of a 1-cycle which represents an element of a basis of 
        $\widetilde{H}_1 (\Gamma;k)$. Let $\Delta$ be the simplicial complex obtained 
        from $\Gamma$ by adding the two-dimensional face $\sigma = \{a, b, c\}$. It is 
        easy to check that $\Delta$ is doubly Buchsbaum over all fields $k$. However, 
        since $\widetilde{H}_1 (\Delta;k) \cong k$ and $\widetilde{H}_1 (|\Delta| - p;k) 
        \cong k^2$ for every point $p$ in the relative interior of $|\sigma|$, the 
        complex $\Delta$ is not Buchsbaum* over $k$.    
    \end{itemize}
    }
  \end{example}

  \begin{corollary} \label{Links}
     If $\Delta$ is a Buchsbaum* simplicial complex over $k$, then $\link_\Delta(\sigma)$ 
     is doubly Cohen-Macaulay over $k$ for every nonempty face $\sigma$ of $\Delta$. 
  \end{corollary}
  \begin{proof}
     This statement follows from Corollary \ref{MainTheorem} and the fact (see, for
     instance, \cite[Lemma 4.2]{Mi}) that the link of any nonempty face in a doubly 
     Buchsbaum complex is doubly Cohen-Macaulay. 
  \end{proof}

    \section{Constructions}
  \label{sec:examples}

  This section investigates the behavior of the Buchsbaum* property under taking products, 
  joins and skeleta of simplicial complexes, studies a notion of higher Buchsbaum* 
  connectivity and shows that a large family of Buchsbaum* complexes can be constructed by 
  gluing orientable homology manifolds with boundary to an orientable homology manifold 
  without boundary in a suitable way.   

  \subsection{Products}

  This section shows that the Buchsbaum and Buchsbaum* properties are 
  preserved under direct products of simplicial complexes. We note that the 
  corresponding statement fails for both the Cohen-Macaulay and doubly 
  Cohen-Macaulay properties. 

  \begin{proposition} \label{Product}
     Let $\Gamma$ be a $(d-1)$-dimensional simplicial complex and $\Delta$ be an 
     $(e-1)$-dimensional simplicial complex. 
     \begin{itemize}
       \item[(i)] 
         If $\Gamma$ and $\Delta$ are Buchsbaum over $k$, 
         then every simplicial complex triangulating 
         $|\Gamma| \times |\Delta|$ is Buchsbaum over $k$.
       \item[(ii)] 
         If $\Gamma$ and $\Delta$ are Buchsbaum* over $k$, 
         then every simplicial complex triangulating 
         $|\Gamma| \times |\Delta|$ is Buchsbaum* over $k$.
     \end{itemize}
  \end{proposition}
  \begin{proof}
     Let $p \in |\Gamma| \times |\Delta|$. There are unique faces
     $\sigma \in \Gamma$ and $\tau \in \Delta$ such that $p$ lies in the 
     relative interior of $|\sigma| \times |\tau|$. Then 
     $|\cost_\Gamma(\sigma)| \times |\Delta| \cup 
     |\Gamma| \times |\cost_\Delta(\tau)|$
     is a deformation retract of $|\Gamma| \times |\Delta| - p$ and hence
     \begin{eqnarray} \label{def}
       (|\Gamma|, |\cost_\Gamma(\sigma)|) \times 
       (|\Delta|,|\cost_\Delta(\tau)|)
       & \simeq & (|\Gamma| \times |\Delta|, |\Gamma| \times |\Delta| -p)
     \end{eqnarray}
     is a deformation retraction.
     \begin{itemize}
       \item[(i)] 
         By (\ref{def}) and the K\"unneth formula we have:
         $$H_i (|\Gamma| \times |\Delta|, |\Gamma| \times |\Delta| -p;k)
                            \cong
           \bigoplus_{j=0}^i 
           \begin{array}{c} 
             H_j (\Gamma, \cost_\Gamma(\sigma);k) \\ 
             \otimes \\
             H_{i-j}(\Delta, \cost_\Delta(\tau);k). 
           \end{array}
         $$  
         For $i < d+e-2$ we have either $j < d-1$ or $i-j < e-1$.
         Thus by Buchsbaumness of $\Gamma$ and $\Delta$, 
         either $H_j (\Gamma, \cost_\Gamma(\sigma);k) = 0$
         or $H_{i-j}(\Delta, \cost_\Delta(\tau);k) = 0$. 
         Thus 
         $$H_i (|\Gamma| \times |\Delta|, |\Gamma| \times |\Delta| -p;k) = 0$$ 
         for $i < d+e-2$. 
         Hence every simplicial complex triangulating $|\Gamma| \times |\Delta|$
         is Buchsbaum over $k$.
       \item[(ii)] 
         By (\ref{def}), the map 
         $$
           \begindc{\commdiag}[3]
             \obj(40,30){$\widetilde{H}_{d+e-2}(|\Gamma|\times |\Delta|;k)$}
             \obj(105,30){$H_{d+e-2}(|\Gamma| \times |\Delta|, 
                                              |\Gamma|\times |\Delta|-p;k)$}
             \mor(55,30)(77,30){$\rho_*$}
           \enddc
         $$

         from the exact sequence of the pair $(|\Gamma| \times |\Delta|,
                                              |\Gamma|\times |\Delta|-p)$
         equals the map 
         $$
         \begindc{\commdiag}[3]
           \obj(40,30){$\widetilde{H}_{d+e-2} (|\Gamma|\times |\Delta|;k)$}
           \obj(80,30){$H_{d+e-2}\Big($}
           \obj(102,34){$(|\Gamma|,|\cost_\Gamma(\sigma)|)$}
           \obj(95,30){$\times$}
           \obj(102,26){$(|\Delta|,|\cost_\Delta(\tau)|)$}
           \obj(120,30){$;k\Big)$.}
           \mor(55,30)(75,30){$\rho_*$}
        \enddc
         $$

         In turn, by the K\"unneth formula, this map can be written as
         $$
         \begindc{\commdiag}[3]
           \obj(40,30){$\widetilde{H}_{d-1} (\Gamma;k) \otimes \widetilde{H}_{e-1} (\Delta;k)$}
           \obj(115,30){$H_{d-1}(\Gamma,\cost_\Gamma(\sigma);k) \otimes 
           H_{e-1}(\Delta,\cost_\Delta(\tau);k)$.}
           \mor(60,30)(80,30){$\rho_*$}
        \enddc
        $$

        From the fact that $\Gamma$ is Buchsbaum* over $k$ and Proposition \ref{Isomorphism} 
        we deduce that the projection from 
        $\widetilde{H}_{d-1}(\Gamma;k)$ to 
        $H_{d-1}(\Gamma,\cost_\Gamma(\sigma);k)$ is surjective.
        Analogously, the projection from 
        $\widetilde{H}_{e-1}(\Delta;k)$ to 
        $H_{e-1}(\Delta,\cost_\Delta(\tau);k)$
        is surjective. Hence $\rho_*$ is surjective as well. This fact and Proposition 
        \ref{Isomorphism} imply that every triangulation of $|\Gamma|\times |\Delta|$ 
        is Buchsbaum* over $k$.
     \end{itemize}
  \end{proof}

  \subsection{Joins}

  We recall that the join $\Gamma * \Delta$ of two simplicial complexes $\Gamma$ and 
  $\Delta$ on disjoint ground sets is the simplicial complex whose faces are the sets 
  of the form $\sigma \cup \tau$, where $\sigma \in \Gamma$ and $\tau \in \Delta$. The 
  following proposition classifies the situations in which $\Gamma * \Delta$ is 
  Buchsbaum*. A similar statement (with a similar proof) holds for the Buchsbaum 
  property.

  \begin{proposition} \label{join}
    Let $\Gamma$ and $\Delta$ be
    simplicial complexes, each having at least one vertex.
    The following are equivalent:
    \begin{itemize}
       \item[(i)] $\Gamma * \Delta$ is Buchsbaum* over $k$.
       \item[(ii)] $\Gamma * \Delta$ is doubly Cohen-Macaulay over $k$.
       \item[(iii)] $\Gamma$ and $\Delta$ are doubly Cohen-Macaulay over $k$.
    \end{itemize}
  \end{proposition}
  \begin{proof}
    \noindent {\sf (i) $\Rightarrow$ (iii):} 
      Since $\Delta$ contains at least one vertex, there exists a nonempty maximal 
      simplex $\sigma \in \Delta$. Since $\Gamma * \Delta$ is Buchsbaum* 
      over $k$ and the link of $\sigma$ in $\Gamma * \Delta$ is equal to
      $\Gamma$, it follows from Corollary \ref{Links} that $\Gamma$ is 
      doubly Cohen-Macaulay over $k$. It follows in a similar way that 
      $\Delta$ is doubly Cohen-Macaulay over $k$.

    \noindent {\sf (iii) $\Rightarrow$ (ii):} 
      It is well known that the join of two Cohen-Macaulay simplicial complexes 
      over $k$ is Cohen-Macaulay over $k$. The implication follows from this 
      statement, the definition of double Cohen-Macaulayness and the fact that 
      for every vertex $v$, say of $\Gamma$, the complex $(\Gamma * \Delta) \sm 
      v$ is equal to the simplicial join $(\Gamma \sm v) * \Delta$.

   \noindent {\sf (ii) $\Rightarrow$ (i):} 
     This follows from Proposition \ref{CMBuchsbaum} (i).
   \end{proof}

  \subsection{Higher Buchsbaum* connectivity and skeleta}
  \label{sec:skeleta}

  Given a subset $\tau$ of the set of vertices of $\Delta$, we denote by
  $\Delta \sm \tau$ the subcomplex $\{\sigma \in \Delta: \sigma \cap \tau =
  \varnothing\}$ of $\Delta$, consisting of all faces of $\Delta$ which do not 
  contain any element of $\tau$. We define a notion of higher Buchsbaum* 
  connectivity for simplicial complexes as follows.  

  \begin{definition} \label{def:mB*}
     Let $\Delta$ be a simplicial complex and let $m$ be a nonnegative integer. 
     We call $\Delta$ \emph{$m$-Buchsbaum*} over $k$ if $m = 0$ and $\Delta$ is 
     Buchsbaum over $k$ or $m \ge 1$ and $\Delta \sm \tau$ is Buchsbaum* over $k$ 
     of the same dimension as $\Delta$ for every set $\tau$  
     of vertices of $\Delta$ of cardinality less than $m$. 
  \end{definition}

  Thus the class of $0$-Buchsbaum* complexes coincides with that 
  of Buchsbaum complexes and the class of $1$-Buchsbaum* 
  complexes coincides with that of Buchsbaum* complexes. Our notion
  of higher connectivity for Buchsbaum* complexes is analogous to that already
  existing for Buchsbaum and Cohen-Macaulay complexes: Given a positive integer 
  $m$, a simplicial complex $\Delta$ is called \emph{$m$-Buchsbaum} over $k$ 
  in \cite{Mi} (respectively, \emph{$m$-Cohen-Macaulay} over $k$ in \cite{Ba}) if 
  $\Delta \sm \tau$ is Buchsbaum over $k$ (respectively, Cohen-Macaulay over $k$) 
  of the same dimension as $\Delta$ for every set $\tau$ of vertices of $\Delta$ 
  of cardinality less than $m$. 

  The following two statements generalize Proposition \ref{CMBuchsbaum} (i) and
  Corollary \ref{MainTheorem}, respectively. 

 \begin{proposition} \label{prop:CMmB*}
     For a Cohen-Macaulay simplicial complex $\Delta$ over $k$ and a nonnegative 
     integer $m$, the following conditions are equivalent: 
     \begin{itemize}
        \item[(i)] $\Delta$ is $m$-Buchsbaum* over $k$. 
        \item[(ii)] $\Delta$ is $(m+1)$-Cohen-Macaulay over $k$. 
     \end{itemize}
  \end{proposition}
  \begin{proof} 
     \noindent {\sf (i) $\Rightarrow$ (ii):} The implication is trivial for $m=0$
     and follows from Proposition \ref{CMBuchsbaum} for $m=1$. We assume that $m \ge 2$
     and proceed by induction on $m$. Suppose that $\Delta$ is $m$-Buchsbaum* over $k$.
     To verify (ii), it suffices to show that $\Delta \sm v$ is $m$-Cohen-Macaulay over 
     $k$ of the same dimension as $\Delta$ for every vertex $v$ of $\Delta$. Indeed, 
     $\Delta$ is doubly Cohen-Macaulay over $k$ by the special case $m=1$ already 
     treated and hence $\Delta \sm v$ is Cohen-Macaulay over $k$ of the same dimension 
     as $\Delta$. Since $\Delta \sm v$ is 
     $(m-1)$-Buchsbaum* over $k$ by Definition \ref{def:mB*}, the desired statement
     follows from the induction hypothesis.  

     \noindent {\sf (ii) $\Rightarrow$ (i):} This follows from part (i) of
     Proposition \ref{CMBuchsbaum} and the relevant definitions.
  \end{proof}

  \begin{proposition} \label{withandwithoutstar}
    Let $m$ be a nonnegative integer and $\Delta$ be a simplicial complex. If $\Delta$ 
    is $m$-Buchsbaum* over $k$, then $\Delta$ is $(m+1)$-Buchsbaum over $k$.
  \end{proposition}
  \begin{proof}
    Let $d-1$ be the dimension of $\Delta$. The statement is a tautology 
    for $m=0$. Assume that $m \ge 1$ and let $\tau$ be a set of vertices 
    of $\Delta$ of cardinality at most $m$. We need to show that $\Delta 
    \sm \tau$ is Buchsbaum over $k$ of dimension $d-1$. This is clear if
    $\tau = \varnothing$. Otherwise, let $v$ 
    be an element of $\tau$ and let $\sigma = \tau \sm \{v\}$ and $\Gamma
    = \Delta \sm \sigma$. The complex $\Gamma$ is Buchsbaum* over $k$ by 
    Definition \ref{def:mB*} and hence it is doubly Buchsbaum over $k$ by 
    Corollary \ref{MainTheorem}. This implies that $\Gamma \sm v$ is 
    Buchsbaum over $k$ of dimension $d-1$. Since $\Gamma \sm v = \Delta 
    \sm \tau$, the latter complex is Buchsbaum over $k$ of dimension $d-1$. 
    This completes the proof.
  \end{proof}  

  Next we show that Buchsbaum* connectivity increases when passing to skeleta. 
  Recall that the \emph{$i$-skeleton} of a simplicial complex $\Delta$ is 
  defined as the simplicial complex $\Delta^{\langle i \rangle}$ of all faces 
  of $\Delta$ of dimension $\le i$. It is known \cite[Corollary 7.6]{Mi} that 
  if $\Delta$ is $(d-1)$-dimensional and Buchsbaum over $k$, then the $i$-skeleton 
  of $\Delta$ is doubly Buchsbaum over $k$ for every $i \le d-2$. In view of 
  Corollary \ref{MainTheorem}, the following is a stronger statement.

  \begin{proposition} \label{1-coskeleton}
     Let $\Delta$ be a $(d-1)$-dimensional simplicial complex which is 
     Buchsbaum over $k$. Then the $i$-skeleton $\Delta^{\langle i \rangle}$ 
     of $\Delta$ is Buchsbaum* over $k$ for every $i \leq d-2$. 
  \end{proposition}
  \begin{proof}
     It suffices to prove the assertion for $i = d-2$. We set $\Gamma = 
     \Delta^{\langle d-2 \rangle}$. For $d = 2$, the $0$-skeleton of any 
     $(d-1)$-dimensional simplicial complex consists of at least two points 
     and therefore it is Buchsbaum* over all fields. 
     Assume that $d \ge 3$. It is known (and follows, for instance, from condition 
     (ii) of Theorem \ref{BuchsbaumCharacterization}) that $\Gamma$ is 
     Buchsbaum over $k$. Thus we only need to check that 
     $\widetilde{H}_{d-3}(|\Gamma|-p;k) \cong 
     \widetilde{H}_{d-3}(|\Gamma|; k)$ 
     for every $p \in |\Gamma|$. By condition (iii) of 
     Theorem \ref{BuchsbaumCharacterization} and the long exact homology 
     sequence for the pair $(|\Delta|, |\Delta| - p)$, we know that
     $\widetilde{H}_{d-3}(|\Delta|-p;k) \cong \widetilde{H}_{d-3}(|\Delta|; k)$
     holds for every $p \in |\Delta|$. Since $p \in |\Gamma| \subseteq |\Delta|$,
     it follows from the fact that the chains groups of $\Gamma$ and $\Delta$
     in simplicial homology coincide in dimensions $\le d-2$ that 
     $\widetilde{H}_{d-3}(|\Gamma|-p;k) \cong 
     \widetilde{H}_{d-3}(|\Delta|-p;k)$ and $\widetilde{H}_{d-3}(|\Gamma|;k)
     \cong \widetilde{H}_{d-3}(|\Delta|;k)$.
     This completes the proof.
  \end{proof}

  The following result extends to Buchsbaum* connectivity analogous statements
  on Buchsbaum \cite[Corollary 7.6]{Mi} and Cohen-Macaulay \cite[Corollary 2.7]{Fl} 
  connectivity.

  \begin{theorem} \label{skeletaconnectivity}
     Let $\Delta$ be a $(d-1)$-dimensional simplicial complex which is
     $m$-Buchsbaum* over $k$. Then the $i$-skeleton $\Delta^{\langle i \rangle}$ 
     is $(m+d-i-1)$-Buchsbaum* over $k$ for every $0 \le i \le d-1$.
  \end{theorem}
  \begin{proof} 
     The statement is a tautology for $i = d-1$. By induction on $d-i-1$, it 
     suffices to show that $\Delta^{\langle d-2 \rangle}$ is $(m+1)$-Buchsbaum* 
     over $k$. Let $\tau$ be any set of vertices of $\Delta$ of cardinality at most 
     $m$ and set $\Gamma = \Delta \sm \tau$. Since $\Delta$ is $(m+1)$-Buchsbaum 
     over $k$ by Proposition \ref{withandwithoutstar}, the complex $\Gamma$ is Buchsbaum 
     over $k$ of dimension $d-1$. It follows from Proposition \ref{1-coskeleton} 
     that $\Gamma^{\langle d-2 \rangle}$ is Buchsbaum* over $k$. Since this 
     skeleton is equal to $\Delta^{\langle d-2 \rangle} \sm 
     \tau$, we conclude that the latter complex is Buchsbaum* over $k$ of 
     dimension $d-2$. Since $\tau$ was arbitrary of cardinality at most $m$,
     the desired statement follows.  
  \end{proof}

  \begin{remark}
  {\rm It is known by the results of Miyazaki \cite{Mi} and Walker \cite{Wa} 
  that double Buchsbaumness and double Cohen-Macaulayness (over a fixed field) 
  are topological properties. It is also known that for $m \ge 3$, neither 
  $m$-Buchsbaumness nor $m$-Cohen-Macaulayness is a topological property. In
  view of Proposition \ref{prop:CMmB*}, it follows that for $m \ge 2$, the 
  $m$-Buchsbaum* condition is not a topological property either. }
  \end{remark}
    
  \begin{remark}
  {\rm A different notion of a higher Buchsbaum* property one may try is the 
  following. Let $\Delta$ be a $(d-1)$-dimensional simplicial complex and let 
  $m$ be a nonnegative integer. We consider the following condition:
  \begin{itemize}
    \item[(M)] $\Delta$ is Buchsbaum over $k$ and 
               $
                \dim_k \widetilde{H}_{d-2}(|\Delta| - P;k)
                = 
                \dim_k \widetilde{H}_{d-2} (\Delta;k) 
               $
               holds for every set $P \subset |\Delta|$ of cardinality at most $m$.
  \end{itemize}
  
  This condition reduces to Buchsbaumness for $m=0$ and to the Buchsbaum* condition 
  for $m = 1$. We claim, however, that no simplicial complex of positive
  dimension satisfies (M) for $m \ge 2$. Clearly, it suffices to show that no such 
  complex satisfies (M) for $m=2$. Suppose on the contrary that $\Delta$ is such a 
  complex. We choose two points $p, q \in |\Delta|$ which lie in the relative interior 
  of some $(d-1)$-dimensional simplex, say $\tau$, of $\Delta$. We triangulate $\tau$ 
  by adding a vertex $z$ and faces $\{z\} \cup \sigma$ for $\sigma \subset \tau$. We 
  realize the new complex in such a way that $p$ and $q$ lie in the relative interior 
  of the realization of two distinct $(d-1)$-dimensional simplices $\tau_p = \{z\} 
  \cup \sigma_p$ and $\tau_q = \{z\} \cup \sigma_q$. We denote by $\Delta'$ the 
  simplicial complex whose simplices are those of $\Delta$ other than $\tau$ and the 
  faces triangulating $\tau$ in the way just described. 
  In particular, $|\Delta' \sm \{\tau_p, \tau_q\}|$ is a deformation retract of 
  $|\Delta| - \{p,q\}$. The assumption that $\Delta$ is Buchsbaum over $k$ and excision 
  give  
  $$0 = H_{d-2}(|\Delta|,|\Delta| -q;k) \cong 
        H_{d-2}(|\Delta|-p,|\Delta| -\{p,q\};k).$$
  The long exact sequence of the triple $(|\Delta|,|\Delta|-p,|\Delta|- \{p,q\})$ then 
  yields that $H_{d-2}(|\Delta|, |\Delta| -\{p,q\};k) = 0$. Thus,
  using the same arguments as in the proof of Proposition 
  \ref{Isomorphism}, it follows from (M) that the inclusion map $\Delta' \sm 
  \{ \tau_p, \tau_q\} \to \Delta'$ induces an isomorphism 
       \begin{equation} \label{eq:Delta'}
         \widetilde{H}_{d-2} (\Delta' \sm \{ \tau_p, \tau_q\};k) \rightarrow 
         \widetilde{H}_{d-2} (\Delta';k). 
       \end{equation} 
  Clearly, $\partial_{d-1} (\tau_p)$ is a boundary in $\Delta'$. Since the map
  (\ref{eq:Delta'}) is an isomorphism, $\partial_{d-1} (\tau_p)$ must be a 
  boundary in $\Delta' \sm \{ \tau_p, \tau_q\}$ as well. However, this is not 
  possible since $\tau_p \cap \tau_q$ is a $(d-2)$-dimensional simplex which 
  lies in the support of $\partial_{d-1} (\tau_p)$ and which is not contained
  in any $(d-1)$-dimensional simplex of $\Delta'$ other than $\tau_p$ and 
  $\tau_q$. This yields the desired contradiction. \qed}
  \end{remark}
  


  \subsection{A generalized convex ear decomposition} 

  In the sequel we describe a class of Buchsbaum* complexes significantly 
  larger than that provided by Proposition \ref{Manifold}. 
  The construction is motivated by and generalizes the convex ear decomposition 
  of simplicial complexes, introduced by Chari \cite{Ch}.

  \begin{theorem} \label{GammaDelta}
     Let $\Gamma$ and $\Delta$ be two simplicial complexes such that:
     \begin{itemize}
       \item[(i)]  $\Gamma$ is $(d-1)$-dimensional and Buchsbaum* over $k$. 
       \item[(ii)]  $\Delta$ is a $(d-1)$-dimensional connected orientable 
          homology manifold over $k$ with boundary $\partial \Delta$ which 
          has the following properties: 
          \begin{itemize}
            \item[(a)]  $\partial \Delta$ is a $(d-2)$-dimensional connected orientable 
               homology manifold over $k$. 
            \item[(b)]  $\partial \Delta = \Gamma \cap \Delta$. 
            \item[(c)]  The inclusion map induces the zero homomorphism 
               $\widetilde{H}_{d-2} (\partial \Delta;k) \rightarrow 
                 \widetilde{H}_{d-2} (\Gamma;k)$. 
          \end{itemize}
     \end{itemize}
     Then $\Gamma \cup \Delta$ is Buchsbaum* over $k$.
  \end{theorem}
  
  As a corollary we obtain an inductive construction as follows.

  \begin{corollary} \label{Class}
     Suppose that $\Delta$ is a $(d-1)$-dimensional simplicial complex and that there 
     exist subcomplexes $\Delta_1, \Delta_2,\dots,\Delta_m$ such that:
     \begin{itemize}
       \item[(i)]  $\Delta = \Delta_1 \cup \Delta_2 \cup \cdots \cup \Delta_m$.
       \item[(ii)]  $\Delta_1$ is a $(d-1)$-dimensional orientable homology 
         manifold over $k$. 
      \item[(iii)]  For $2 \le i \le m$, $\Delta_i$ is a $(d-1)$-dimensional 
         connected orientable homology manifold over $k$ with boundary 
         $\partial \Delta_i$ which has the following properties: 
         \begin{itemize}
           \item[(a)]  $\partial \Delta_i$ is a $(d-2)$-dimensional connected 
              orientable homology manifold over $k$. 
           \item[(b)]  $\partial \Delta_i = \Delta_i \cap (\Delta_1 \cup 
                            \cdots \cup \Delta_{i-1})$. 
           \item[(c)]  The inclusion map induces the zero homomorphism
              $$\widetilde{H}_{d-2} (\partial \Delta_i;k) \rightarrow 
              \widetilde{H}_{d-2} (\Delta_1 \cup \cdots \cup \Delta_{i-1};k). $$    
         \end{itemize}
     \end{itemize}
     Then $\Delta$ is Buchsbaum* over $k$.
  \end{corollary}
  \begin{proof}
    The complex $\Delta_1$ is Buchsbaum* over $k$ by Proposition \ref{Manifold}. 
    The theorem follows from this statement and Theorem \ref{GammaDelta} by 
    induction on $m$. 
  \end{proof}

  \begin{proof}[Proof of Theorem \ref{GammaDelta}]
     Since $\Gamma$ and $\Delta$ are Buchsbaum over $k$ of dimension $d-1$ and 
     $\Gamma \cap \Delta$ is Buchsbaum over $k$ of dimension $d-2$, it follows 
     by a standard argument (used, for instance, in the proof of \cite[Lemma 1]{BjH}) 
     that $\Gamma \cup \Delta$ is also Buchsbaum over $k$. We consider a point $p \in 
     |\Gamma \cup \Delta|$. To show that (\ref{eq:Buch*}) (or the equivalent condition
     (iv) of Proposition \ref{Isomorphism}) holds for $\Gamma \cup \Delta$, we 
     distinguish three cases.  

     \medskip

     \noindent {\sf Case 1:} $p \in |\Gamma| \sm |\Delta|$. The naturality of the 
        long exact homology sequence for pairs gives the commutative diagram
  
        \medskip

        $$
          \begindc{\commdiag}[3]
            \obj(0,30){ $\widetilde{H}_{d-1} (|\Gamma|;k)$ }
            \obj(55,30){$H_{d-1} (|\Gamma|,|\Gamma|-p;k)$.}
            \mor(15,30)(31,30){$\rho_*$}
            \obj(0,50){$\widetilde{H}_{d-1} (|\Gamma \cup \Delta|;k)$ }
            \obj(56,50){$H_{d-1} (|\Gamma \cup \Delta|,
                               |\Gamma \cup \Delta| - p;k)$}
            \mor(14,50)(31,50){$\widetilde{\rho}_*$}
            \mor(55,32)(55,48){$\epsilon_*$}
            \mor(0,32)(0,48){$~$}
         \enddc
       $$

       \medskip

       \noindent Since $\Gamma$ is Buchsbaum* over $k$, the map $\rho_*$ is surjective.
       The map $\epsilon_*$ is an excision map and hence an isomorphism. The 
       commutativity of the diagram implies that $\widetilde{\rho}_*$ is surjective as 
       well.

     \medskip

     \noindent {\sf Case 2:} $p \in |\Delta| \sm |\Gamma|$. The long exact homology 
     sequences for the triples  $(|\Delta|, |\Delta| - p, \partial |\Delta|)$ and 
     $(|\Gamma \cup \Delta|, |\Gamma \cup \Delta| - p, \partial |\Delta| )$ and the 
     naturality of such sequences yield the following commutative diagram:

       $$
         \begindc{\commdiag}[3] 
           \obj(0,50){$0$}
           \obj(31,50){$H_{d-1} (|\Gamma \cup \Delta| - p, 
                                               \partial |\Delta|;k)$}

           \obj(78,50){$H_{d-1} (|\Gamma \cup \Delta|, 
                                               \partial |\Delta|;k)$}
           \obj(128,50){$H_{d-1}(|\Gamma \cup \Delta|, 
                                            |\Gamma \cup \Delta| - p;k)$}
           \obj(31,30){$H_{d-1} (|\Delta| - p, \partial |\Delta|;k)$}
           \obj(78,30){$H_{d-1} (|\Delta|, \partial |\Delta|;k)$}
           \obj(127,30){$H_{d-1}(|\Delta|, |\Delta| - p;k)$.}
          
           \mor(0,50)(11,50){$~$}
           \mor(51,50)(62,50){$~$}
           \mor(95,50)(106,50){$\widetilde{\delta}_*$}
           \mor(51,30)(63,30){$~$} 
           \mor(32,32)(32,48){$~$}
           \mor(80,32)(80,48){$~$}
           \mor(126,32)(126,48){$\epsilon_*$}
           \mor(94,30)(108,30){$\delta_*$}
         \enddc
       $$

       \medskip

       \noindent Our assumptions on $\Delta$ and $p$ imply that $\widetilde{H}_{d-1} 
       (|\Delta| - p, \partial |\Delta|;k) = 0$ and $H_{d-1} (|\Delta|, 
       \partial |\Delta|;k) \cong H_{d-1} (|\Delta|, |\Delta|-p;k) \cong k$. 
       It follows that the map $\delta_*$ is an isomorphism. Since $\epsilon_*$ is an 
       isomorphism by excision, the commutativity 
       of the square on the right implies that the map $\widetilde{\delta}_*$ is 
       surjective. These facts and the exactness of the top row of the diagram imply 
       that 
         \begin{equation} \label{Dim1}
           \dim_k H_{d-1}(|\Gamma \cup \Delta|, \partial |\Delta|;k) \ = 
           \ \dim_k H_{d-1}(|\Gamma \cup \Delta| - p, \partial |\Delta|; k)
           \, + \, 1.  
         \end{equation}
       Our assumption ${\rm (c)}$ and the long exact homology sequences for 
       the pairs $(|\Gamma \cup \Delta|, \partial |\Delta|)$ and $(|\Gamma \cup \Delta|
       - p, \partial |\Delta|)$ yield exact sequences of the form

       $$
         \begindc{\commdiag}[3]
           \obj(0,30){$0$}
           \obj(25,30){$\widetilde{H}_{d-1} (|\Gamma \cup \Delta|;k)$}
           \obj(65,30){$H_{d-1}(|\Gamma \cup \Delta|, 
                                   \partial |\Delta|;k)$}
           \obj(102,30){$\widetilde{H}_{d-2}(\partial |\Delta|;k)$}
           \obj(124,30){$0$}
           \mor(0,30)(13,30){$~$}
           \mor(37,30)(48,30){$~$}
           \mor(82,30)(93,30){$~$}
           \mor(112,30)(125,30){$~$}
         \enddc
       $$

       and 

       $$
         \begindc{\commdiag}[3]
           \obj(0,30){$0$}
           \obj(29,30){$\widetilde{H}_{d-1} (|\Gamma \cup \Delta| - p;k)$}
           \obj(76,30){$H_{d-1}(|\Gamma \cup \Delta| - p, 
                                   \partial |\Delta|;k)$}
           \obj(117,30){$\widetilde{H}_{d-2}(\partial |\Delta|;k)$}
           \obj(138,30){$0$.}
           \mor(0,30)(13,30){$~$}
           \mor(45,30)(56,30){$~$}
           \mor(96,30)(107,30){$~$}
           \mor(127,30)(138,30){$~$}
         \enddc
       $$

       \noindent These sequences and (\ref{Dim1}) imply that
         \begin{equation} \label{Dim2}
           \dim_k \widetilde{H}_{d-1}(|\Gamma \cup \Delta|;k) \ = 
           \ \dim_k \widetilde{H}_{d-1}(|\Gamma \cup \Delta| - p; k) \, + \, 1.  
         \end{equation} 
       Finally, we consider the exact sequence 
       $$
         \begindc{\commdiag}[3] 
           \obj(0,50){$0$}
           \obj(28,50){$\widetilde{H}_{d-1} (|\Gamma \cup \Delta| - p;k)$}
           \obj(68,50){$\widetilde{H}_{d-1} (|\Gamma \cup \Delta|;k)$}
           \obj(118,50){$H_{d-1}({{|\Gamma \cup \Delta|,} {|\Gamma \cup \Delta| - p}};k)=k$,}
            \mor(0,50)(11,50){$~$}
            \mor(44,50)(55,50){$~$}
            \mor(80,50)(91,50){$\widetilde{\rho}_*$}
         \enddc
       $$
       coming from the long exact homology sequence for the pair $(|\Gamma \cup \Delta|, 
       |\Gamma \cup |\Delta| - p)$, where $H_{d-1} (|\Gamma \cup \Delta|, 
       |\Gamma \cup \Delta| - p;k) \cong H_{d-1} (|\Delta|, |\Delta| - p;k) 
       \cong k$. Equation (\ref{Dim2}) shows that $\widetilde{\rho}_*$ cannot be the 
       zero map. This forces $\widetilde{\rho}_*$ to be surjective.
 
    \medskip

     \noindent {\sf Case 3:} $p \in |\Gamma \cap \Delta| = \partial |\Delta|$. 
     From the Mayer-Vietoris sequence for the pairs $(|\Delta|, |\Gamma|)$ 
     and $(|\Delta|-p, |\Gamma|-p)$ and the naturality of such sequences, we get 
     a commutative diagram as follows:

       $$
         \begindc{\commdiag}[2]
           \obj(55,17){$\widetilde{H}_{d-2}(|\Gamma|-p;k)$}
           \obj(55,10){$\oplus$}
           \obj(55,3){$\widetilde{H}_{d-2}(|\Delta|-p;k)$}
           \obj(115,10){$\widetilde{H}_{d-2}(|\Gamma \cup \Delta|-p;k)$}
           \obj(178,10){$\widetilde{H}_{d-3}(\partial |\Delta|-p;k)$.}
           \obj(0,50){$\widetilde{H}_{d-2} (\partial |\Delta|;k)$}
           \obj(55,57){$\widetilde{H}_{d-2}(|\Gamma|;k)$}
           \obj(55,50){$\oplus$}
           \obj(55,43){$\widetilde{H}_{d-2}(|\Delta|;k)$}
           \obj(115,50){$\widetilde{H}_{d-2}(|\Gamma \cup \Delta|;k)$}
           \obj(178,50){$\widetilde{H}_{d-3}(\partial |\Delta|;k)$}
           \mor(138,10)(158,10){}
           \mor(136,50)(160,50){}
           \mor(16,50)(40,50){$\alpha_*$}
           \mor(70,10)(92,10){$~$}
           \mor(71,50)(95,50){$~$}
           \mor(55,16)(55,40){$\beta_*$}
           \mor(115,12)(115,48){$\widetilde{\iota}_*$}
           \mor(178,12)(178,48){$\gamma_*$}
         \enddc
       $$

       \medskip

       \noindent
       We claim that the map $\alpha_*$, induced by the inclusions of $\partial 
       |\Delta|$ into $|\Gamma|$ and $|\Delta|$, is the zero map. Indeed, for the 
       map $\widetilde{H}_{d-2} (\partial |\Delta|; k) \rightarrow \widetilde{H}_{d-2}
       (|\Gamma|; k)$ induced by inclusion, this holds by condition (c). For the other 
       map, consider the exact sequence 

       $$
         \begindc{\commdiag}[3]
           \obj(0,30){$\widetilde{H}_{d-1} (|\Delta|;k)$}
           \obj(40,30){$H_{d-1}(|\Delta|, \partial |\Delta|;k)$}
           \obj(80,30){$\widetilde{H}_{d-2}(\partial |\Delta|;k)$}
           \obj(115,30){$\widetilde{H}_{d-2}(|\Delta|;k)$,}
           \mor(10,30)(25,30){$$}
           \mor(54,30)(69,30){$\partial_*$}
           \mor(91,30)(106,30){$j_*$}
         \enddc
       $$

       \medskip

       \noindent 
       obtained from the long exact homology sequence of the pair 
       $(|\Delta|, \partial |\Delta|)$. Since $\Delta$ is a $(d-1)$-dimensional 
       connected orientable homology manifold with nonempty boundary $\partial \Delta$, 
       we have $\widetilde{H}_{d-1} (|\Delta|;k) = 0$ and $H_{d-1}(|\Delta|, 
       \partial |\Delta|;k) \cong k$. Since $\partial \Delta$ is a $(d-2)$-dimensional 
       connected orientable homology manifold without boundary, we have 
       $\widetilde{H}_{d-2}(\partial |\Delta|;k) \cong k$. It follows that
       $\partial_*$ is an isomorphism and hence that $j_* : \widetilde{H}_{d-2} (\partial 
       |\Delta|;k) \rightarrow \widetilde{H}_{d-2} (|\Delta|;k)$ is the zero map. 

       Finally, we note that the maps $\widetilde{H}_{d-2}(|\Gamma|-p; k) 
       \rightarrow \widetilde{H}_{d-2}(|\Gamma|; k)$, 
       $\widetilde{H}_{d-2}(|\Delta|-p; k) \rightarrow 
       \widetilde{H}_{d-2}(|\Delta|; k)$ and 
       $\widetilde{H}_{d-3}(\partial |\Delta|-p; k) \rightarrow 
       \widetilde{H}_{d-3}(\partial |\Delta|; k)$, induced by inclusions, 
       are isomorphisms. This follows by our assumption (i) in the first 
       case, by condition (a) of our assumption (ii) and Proposition 
       \ref{Manifold} in the third case and by the long exact homology 
       sequence of the pair $(|\Delta|,|\Delta|-p)$ in the second case, 
       if one observes that $p \in \partial |\Delta|$ implies 
       $H_i (|\Delta|, |\Delta|-p;k) = 0$ for all $i$.  
       Thus the vertical maps $\beta_*$ and $\gamma_*$ in the diagram are 
       isomorphisms.

       Given the above, it follows by diagram chasing that the map 
       $\widetilde{\iota}_*: \widetilde{H}_{d-2}(|\Gamma \cup \Delta| - p; k) 
       \rightarrow \widetilde{H}_{d-2}(|\Gamma \cup \Delta|; k)$ is
       injective and hence an isomorphism. Thus (\ref{eq:Buch*}) holds in 
       this case as well.
  \end{proof}

  \begin{remark}
  {\rm Example \ref{ex:counter} (iii) shows that condition (c) in Theorem  
  \ref{GammaDelta} and Corollary \ref{Class} cannot be dropped.}
  \end{remark}

  \section{The graph of a Buchsbaum* complex}
  \label{sec:graph}

  The graph of a simplicial complex $\Delta$ is defined as the abstract graph 
  $\gG(\Delta)$ whose nodes are the vertices of $\Delta$ and whose edges are the 
  one-dimensional simplices. This section shows that a result of Nevo \cite{Ne} 
  on the rigidity of graphs of doubly Cohen-Macaulay complexes extends easily to 
  those of connected Buchsbaum* complexes. Since the proof follows from that of 
  \cite{Ne} by minor modifications, we will only indicate those points in the 
  proof where some modification is actually needed.   

  Let $\gG$ be an abstract graph (without loops or multiple edges) on the set 
  of nodes $V$ and let $\|x\|$ denote the Euclidean length of $x \in \RR^d$. 
  A map $f: V \to \RR^d$ is called \emph{$\gG$-rigid} if there exists 
  $\varepsilon > 0$ with the following property: if $g: V \to \RR^d$ is a map 
  satisfying $\|f(v) - g(v)\| < \varepsilon$ for every $v \in V$ 
  and $\|g(u) - g(v)\| = \|f(u) - f(v)\|$ for every edge $\{u, v\}$ of $\gG$, 
  then we have $\|g(u) - g(v)\| = \|f(u) - f(v)\|$ for all $u, v \in V$. 
  The graph $\gG$ is called \textit{generically $d$-rigid} if the set of all 
  $\gG$-rigid maps $f: V \to \RR^d$ is open and dense in the topological 
  vector space of all maps $f: V \to \RR^d$. 

  \begin{theorem} \label{graph}
     Let $\Delta$ be a $(d-1)$-dimensional connected simplicial complex which 
     is Buchsbaum*  over some field $k$. If $d \ge 3$, then the graph $\gG(\Delta)$ is 
     generically $d$-rigid. 
  \end{theorem}
  \begin{proof}
    As in the proof of \cite[Theorem 1.3]{Ne}, it suffices to show that for 
    $d \ge 2$ every such complex $\Delta$ admits a decomposition into minimal 
    $(d-1)$-cycle complexes, as in \cite[Theorem 3.4]{Ne}. 
    The case $d=2$ is covered by \cite[Theorem 3.4]{Ne}, since every 
    one-dimensional connected Buchsbaum* complex is doubly Cohen-Macaulay 
    (see Example \ref{ex:1dim}). Thus we may assume that $d \ge 3$ and proceed 
    by induction on $d$. Let $v$ be any vertex of $\Delta$ and note
    that the induction hypothesis applies to the complex $\link_\Delta(v)$, 
    which is doubly Cohen-Macaulay over $k$ by Corollary \ref{Links}. Note also
    that if $s$ is a minimal $(d-2)$-cycle for $\link_\Delta(v)$, then there 
    exists a $(d-1)$-chain $c$ for $\Delta \sm v$ such that $\partial_{d-1} (c) = s$.
    Indeed, this follows from condition (ii) of Proposition \ref{Isomorphism}, 
    since $s$ is trivial as an element of $\widetilde{H}_{d-2}(\Delta; k)$.
    The remainder of the proof follows that of \cite[Theorem 3.4]{Ne} without 
    change.     
  \end{proof}

  Given a positive integer $m$, an abstract graph $\gG$ is said to be 
  \emph{$m$-connected} if $\gG$ has at least $m+1$ nodes and any graph obtained from 
  $\gG$ by deleting $m-1$ or fewer nodes and their incident edges is connected 
  (necessarily with at least one edge). Part (i) of the following corollary is a 
  special case of a result independently found by Bj\"orner \cite{Bj-new}. Example 
  \ref{ex:counter} (i) shows that part (ii) is not valid if Buchsbaum* is replaced 
  by doubly Buchsbaum.

  \begin{corollary} \label{Connectivity}
     Let $\Delta$ be a connected simplicial complex of dimension $d-1 \ge 1$. 
     \begin{itemize}
       \item[(i)] If $\Delta$ is Buchsbaum over some field, then the graph 
                  $\gG(\Delta)$ is $(d-1)$-connected.
       \item[(ii)] If $\Delta$ is Buchsbaum* over some field, then the graph 
                  $\gG(\Delta)$ is $d$-connected.
     \end{itemize}
  \end{corollary}
  \begin{proof} 
  Proposition \ref{prop:CMmB*} implies that a one-dimensional connected simplicial 
  complex $\Gamma$ is $m$-Buchsbaum* if and only if $\Gamma$ is $(m+1)$-connected as 
  a graph. Therefore, both parts of the corollary follow from the special case $i=1$ 
  of Theorem \ref{skeletaconnectivity}. Part (ii) also follows from Theorem \ref{graph} 
  for $d \ge 3$, since every generically $d$-rigid graph is $d$-connected, and from 
  the discussion in Example \ref{ex:1dim} for $d = 2$. 
  \end{proof}

  \section{Face enumeration}
  \label{sec:enumeration}

  This section is concerned with enumerative properties of Buchsbaum* complexes.
  Let $\Delta$ be a $(d-1)$-dimensional simplicial complex. For $0 \le i \le d$, we 
  denote by $\fvec_{i-1} (\Delta)$ the number of faces of $\Delta$ of dimension $i-1$ 
  (in particular, we have $\fvec_{-1} (\Delta) = 1$ unless $\Delta = \varnothing$). 
  The $\hvec$-vector of $\Delta$ is the sequence $\hvec(\Delta) = (\hvec_0 (\Delta), 
  \hvec_1 (\Delta),\dots,\hvec_d (\Delta))$ defined by 
    \begin{eqnarray} \label{hdefinition} 
    \hvec_j (\Delta) & = & \sum_{i=0}^j \, (-1)^{j-i} {d-i \choose d-j} \fvec_{i-1}
    (\Delta).
    \end{eqnarray}      
  We refer the reader to \cite[Chapter II]{Sta} for the importance of this concept and
  recall that the numbers $\hvec_j (\Delta)$ are nonnegative integers, if $\Delta$ is 
  Cohen-Macaulay over $k$. 

  The $\phvec$-vector of $\Delta$ is the sequence $\phvec(\Delta) = (\phvec_0 (\Delta), 
  \phvec_1 (\Delta),\dots,\phvec_d (\Delta))$ defined by 
    \begin{eqnarray} \label{hprimdefinition} 
    \phvec_j (\Delta) & = & \hvec_j(\Delta) + {d \choose j} \, \sum_{i=0}^{j-1} \,
                  (-1)^{j-i-1} \wbeta_{i-1}(\Delta),
    \end{eqnarray}      
  where $\wbeta_{i-1}(\Delta) = \dim_k  \widetilde{H}_{i-1} (\Delta;k)$. We note that
  $\phvec_d (\Delta) = \wbeta_{d-1} (\Delta)$ and that if $\Delta$ is Cohen-Macaulay over 
  $k$, then $\phvec (\Delta) = \hvec (\Delta)$. It was proved by Schenzel \cite{Sch} that 
  if $\Delta$ is Buchsbaum over an infinite field $k$, then
  \begin{eqnarray} \label{Parameters}
    \phvec_i (\Delta) \ = \ \dim_k \left( k[\Delta] / (\Theta_1,\dots,\Theta_d)
    \right)_i
  \end{eqnarray}
  for $0 \le i \le d$, where $(\Theta_1,\dots,\Theta_d)$ is a linear system of parameters 
  for $k[\Delta]$ (here we denote by $A_i$ the $i$th graded component of a graded 
  algebra $A$). Thus if $\Delta$ is Buchsbaum over $k$, then the $\phvec (\Delta)$ are 
  nonnegative integers which may depend on the characteristic of $k$. 

  The following proposition gives an analogue for the $\phvec$-vector of a Buchsbaum* 
  complex to a well-known recursive formula for the $\hvec$-vector of a pure simplicial 
  complex; see (\ref{DeletionContraction}) below. We will write $\Delta \slash v = 
  \{ \sigma \sm \{v\}: \sigma \in \Delta, \, v \in \sigma\}$ for the link of a vertex 
  $v$ of $\Delta$. 

  \begin{proposition} \label{Deletion}
    Let $\Delta$ be a $(d-1)$-dimensional Buchsbaum* simplicial complex over 
    $k$. For each vertex $v$ of $\Delta$ and $0 \le j \le d$ we have
      \begin{eqnarray} \label{Edeletion} 
      \phvec_j (\Delta) & = & \phvec_j(\Delta \sm v) + \hvec_{j-1} (\Delta / v),
      \end{eqnarray}      
    where $\hvec_{-1}(\Delta / v) = 0$ by convention. 
  \end{proposition}
  \begin{proof}
     Applied to $\Delta \sm v$, equation \eqref{hprimdefinition} yields
     \begin{eqnarray} \label{appliedtodeletion} 
        \phvec_j(\Delta \sm v) & = & \hvec_j (\Delta \sm v) + {d \choose j} \, 
        \sum_{i=0}^{j-1} \, (-1)^{j-i-1} \wbeta_{i-1}(\Delta \sm v).
     \end{eqnarray}      
     It is well known (see, for instance, \cite[Lemma 4.1]{Ath}) that 
     \begin{eqnarray} \label{DeletionContraction} 
        \hvec_j(\Delta) & = & \hvec_j(\Delta \sm v) + \hvec_{j-1}(\Delta \slash v)
     \end{eqnarray}
     holds for $0 \le j \le d$.
     Combining equations \eqref{appliedtodeletion} and \eqref{DeletionContraction}, we 
     get
     \begin{eqnarray} \label{conclusion} 
        \phvec_j(\Delta \sm v) + \hvec_{j-1} (\Delta / v) & = & 
        \hvec_j(\Delta) + {d \choose j} \, \sum_{i=0}^{j-1} \, (-1)^{j-i-1} 
        \wbeta_{i-1} (\Delta \sm v).
     \end{eqnarray}
     Since $\Delta$ is Buchsbaum* over $k$, we have $\wbeta_{i-1} (\Delta \sm v) = 
     \wbeta_{i-1} (\Delta)$ for $i \le d-1$. Hence the right-hand side of 
     \eqref{conclusion} is equal to that of \eqref{hprimdefinition} and the result
     follows.
  \end{proof}

  \medskip
  The $\pphvec$-vector of $\Delta$ is the sequence $\pphvec(\Delta) = (\pphvec_0 (\Delta), 
  \pphvec_1 (\Delta),\dots,\pphvec_d (\Delta))$ defined by 
    \begin{eqnarray} \label{hppdefinition} 
      \pphvec_j (\Delta) & = & \phvec_j (\Delta) - {d \choose j} \wbeta_{j-1}(\Delta) 
      \ = \ \hvec_j (\Delta) + {d \choose j} \, \sum_{i=0}^{j} \, (-1)^{j-i-1} 
      \wbeta_{j-1}(\Delta)
    \end{eqnarray}      
  for $0 \le j \le d-1$ and $\pphvec_d (\Delta) = \wbeta_{d-1} (\Delta) = \phvec_d 
  (\Delta)$. The $\pphvec$-vector was introduced by Kalai (see \cite[Section 7]{No}) as 
  the ``correct" $\hvec$-vector for orientable homology manifolds and shown to have 
  nonnegative entries for every Buchsbaum simplicial complex $\Delta$ over $k$ in 
  \cite[Theorem 3.4]{NS2}. 
    \begin{proposition} \label{Deletion2}
      Let $\Delta$ be a $(d-1)$-dimensional Buchsbaum* simplicial complex over 
      $k$. For each vertex $v$ of $\Delta$ and $0 \le j \le d$ we have
        \begin{eqnarray} \label{Edeletion2} 
          \pphvec_j (\Delta) & = & \pphvec_j (\Delta \sm v) + \hvec_{j-1} (\Delta / v),
        \end{eqnarray}      
      where $\hvec_{-1}(\Delta / v) = 0$ by convention. 
    \end{proposition}
    \begin{proof}
      The proposed equation is equivalent to (\ref{Edeletion}) for $j = d$ and follows 
      as in the proof of Proposition \ref{Deletion} for $0 \le j \le d-1$.
    \end{proof}

  Recall that a simplicial complex $\Delta$ is called \emph{flag} if every minimal 
  non-face of $\Delta$ has at most two elements. As an application of Proposition 
  \ref{Deletion}, we will show (Corollary \ref{Flag}) that among all Buchsbaum* flag 
  simplicial complexes of dimension $d-1$, the simplicial join of $d$ copies of the 
  zero-dimensional sphere has the minimum $\phvec$-vector. This result generalizes 
  one of \cite{Ath} on Cohen-Macaulay complexes to the setting of Buchsbaum complexes. 
  The question of formulating such a generalization provided the initial motivation 
  for introducing the class of Buchsbaum* complexes. 

  The following proposition extends \cite[Theorem 2.1]{Sta2} (see also \cite[Theorem 
  9.1]{Sta}) in the setting of Buchsbaum complexes.

  \begin{proposition} \label{Monotonicity}
    Let $\Delta$ be a simplicial complex of dimension $d-1$ and $\Gamma$ be a subcomplex
    of dimension $e-1$. Assume that no set of $e+1$ vertices of $\Gamma$ is a face of 
    $\Delta$ (this condition holds automatically if $d=e$). If both $\Gamma$ and $\Delta$
    are Buchsbaum over $k$, then $\phvec_i (\Gamma) \le \phvec_i (\Delta)$ holds for all 
    $0 \le i \le d$. 
  \end{proposition}
  \begin{proof}
    By equation (\ref{Parameters}), a proof of the proposition can be given by simply 
    replacing the term $\hvec$-vector with $\phvec$-vector in the proof of \cite[Theorem 
    2.1]{Sta2}. 
  \end{proof}

  \begin{corollary} \label{Flag}
     If $\Delta$ is a $(d-1)$-dimensional flag simplicial complex which is Buchsbaum* 
     over $k$, then the inequalities 
       \begin{equation} \label{eq:flag}
       \phvec_i (\Delta) \ge {d \choose i}
       \end{equation} 
     hold for $0 \le i \leq d$.  
  \end{corollary}
  \begin{proof}
     In view of Propositions \ref{Deletion} and \ref{Monotonicity}, this follows 
     by replacing $\hvec$-vectors by $\phvec$-vectors in the argument of \cite[Section 
     4]{Ath} and using the fact (Corollary \ref{MainTheorem}) that $\Delta$ is doubly 
     Buchsbaum over $k$.  
  \end{proof}

  We note that by the result of \cite{NS2} on the nonnegativity of $\pphvec(\Delta)$, 
  mentioned earlier, we have $\phvec_i (\Delta) \ge {d \choose i} \wbeta_{i-1}(\Delta)$ 
  for every $(d-1)$-dimensional Buchsbaum complex $\Delta$.   

  \medskip
  We conclude this section with two results on the face enumeration of Buchsbaum* 
  complexes. They both extend results of Nevo \cite{Ne} on doubly Cohen-Macaulay 
  complexes and rely heavily on Theorem \ref{graph}. 

  \begin{proposition} \label{prop:LBT}
     Let $\Delta$ be a $(d-1)$-dimensional Buchsbaum* simplicial complex over 
     some field and let $n$ be the number of vertices of $\Delta$. If $d \ge 3$, 
     then $\fvec_i (\Delta) \ge \fvec_i (n, d)$ holds for $0 \le i \le d-1$, 
     where $$ \fvec_i (n, d) = \begin{cases}
        {d \choose i} n - {d+1 \choose i+1} i,&\text{if \ $0 \le i \le d-2$} \\
        (d-1)n - (d+1)(d-2), & \text{if \ $i=d-1$}
     \end{cases}$$
     is the number of $i$-dimensional faces of a stacked $(d-1)$-dimensional sphere 
     with $n$ vertices.
  \end{proposition}
  \begin{proof}
     The assertion follows from Theorem \ref{graph} and the discussion 
     in \cite[Section 1]{Ne}. 
  \end{proof}

  The $\gvec$-vector of $\Delta$ is the sequence $(\gvec_0 (\Delta), \gvec_1 
  (\Delta),\dots, \gvec_{\lfloor d/2 \rfloor} (\Delta))$, defined by $\gvec_i (\Delta) 
  = \hvec_i (\Delta) - \hvec_{i-1} (\Delta)$ for $i \ge 1$ and $\gvec_0 (\Delta) = 
  \hvec_0 (\Delta) = 1$. Recall that a sequence $(a_0, \ldots, a_r)$ of nonnegative 
  integers is called an \emph{$M$-vector} if there exists a standard graded $k$-algebra 
  $A = A_0 \oplus \cdots \oplus A_r$ such that $\dim_k A_i = a_i$ for every $i$; see 
  \cite[Section 4.2]{BH} for details and for further information. 

  \begin{proposition} \label{prop:g}
     Let $\Delta$ be a connected $(d-1)$-dimensional simplicial complex which is 
     Buchsbaum* over some field. If $d \ge 4$, then $(\gvec_0 (\Delta), \gvec_1 
     (\Delta), \gvec_2 (\Delta))$ is an $M$-vector.
  \end{proposition}
  \begin{proof}    
     Since $\Delta$ is connected, we have $\phvec_i (\Delta) = \hvec_i (\Delta)$ 
     for $i \le 2$. Therefore (\ref{Parameters}) continues to hold for $i \le 2$,
     if $\phvec_i (\Delta)$ is replaced by $\hvec_i (\Delta)$. Thus the assertion 
     follows from Theorem \ref{graph} as in the discussion in \cite[Section 2]{Ne}.
  \end{proof}

  The following question is an extension to Buchsbaum* complexes of a question 
  posed by A. Bj\"orner and E. Swartz (see \cite[Problem 4.2]{Sw}) for doubly 
  Cohen-Macaulay complexes. An example pointed out by E. Swartz (personal 
  communication with the authors) shows that for $d = 7$, the analogous question 
  with the numbers $\pphvec_i (\Delta)$ replaced by the $\pphvec_i (\Delta)$ has 
  a negative answer.

  \begin{question} \label{que:phh}
    Are the following true for every $(d-1)$-dimensional Buchsbaum* simplicial complex 
    $\Delta$ over $k$? 
    \begin{itemize}
       \item[(i)] $\pphvec_i (\Delta)  \le \pphvec_{d-i}(\Delta)$ for $0 \le i \le 
       \lfloor d/2 \rfloor$.
       \item[(ii)] $(\ppgvec_0(\Delta),\dots,\ppgvec_{\lfloor d/2 \rfloor} (\Delta) )$ 
       is an $M$-vector, where $\ppgvec_i(\Delta) = \pphvec_i(\Delta) - \pphvec_{i-1}
       (\Delta)$ for $i \ge 1$ and $\ppgvec_0 (\Delta) = 1$.
    \end{itemize}
  \end{question}

  \section{Further results}
  \label{sec:questions}

  This section summarizes some interesting results which appeared in the literature 
  after this paper was publicized, giving further important properties of 
  Buchsbaum* complexes.

  \medskip
  \noindent
  1. Let $\Delta$ be a $(d-1)$dimensional Buchsbaum simplicial complex over an 
  infinite field $k$ and let $\Theta = (\Theta_1, \Theta_2,\dots,\Theta_d)$ be a linear 
  system of parameters for $k[\Delta]$. Consider the ring $\overline{k[\Delta]} = 
  (k[\Delta]/\Theta)/I$, where $I = \bigoplus_{i=1}^{d-1} \, {\rm Soc} 
  (k[\Delta]/\Theta)_i$. It was proved by Novik and Swartz in \cite[Theorem 3.5]{NS2} 
  that we have
    \begin{equation} \label{eq:NS}
      \dim_k \, ({\rm Soc} (k[\Delta]/\Theta))_i \, \ge \, {d \choose i} \wbeta_{i-1}
      (\Delta) 
    \end{equation} 
  or, equivalently,
    \begin{equation} \label{eq:NS2}
      \pphvec_i (\Delta) \, \ge \, \dim_k \, \overline{k[\Delta]}_i,
    \end{equation} 
  for $0 \le i \le d$ and in \cite[Theorem 1.3]{NS} that equality holds in (\ref{eq:NS}) 
  if $\Delta$ is a connected orientable homology manifold over $k$. Nagel \cite[Theorem 
  1.1]{Na} showed that equality holds in (\ref{eq:NS}) for every Buchsbaum* complex 
  $\Delta$ over $k$ and all $0 \le i \le d$. He also showed \cite[Theorem 1.2]{Na} that 
  if $\Delta$ is Buchsbaum* over $k$, then $\overline{k[\Delta]}$ is a level ring of 
  Cohen-Macaulay type $\wbeta_{i-1} (\Delta)$ and socle degree $d$. This result provides 
  an analogue to \cite[Theorem 1.4]{NS}, which states that $\overline{k[\Delta]}$ is a 
  Gorenstein ring for every connected orientable homology manifold $\Delta$ over $k$.  

  \medskip
  \noindent
  2. One of the standard constructions on simplicial complexes is rank selection
  on balanced complexes; see \cite[Section III.4]{Sta} for an exposition of 
  these concepts. Responding to a question raised in an earlier version of this paper,
  Browder and Klee \cite{BK} have shown that if a balanced simplicial complex $\Delta$ 
  is Buchsbaum* over $k$, then so is every rank selected subcomplex of $\Delta$. Among
  other results, they have also generalized the inequalities (\ref{eq:flag}) to 
  $m$-Buchsbaum* complexes and treated the case of equality. 

  \medskip
  \noindent
  3. It was shown by I. Novik (personal communication with the authors) that, 
  under the assumptions of Corollary \ref{Flag}, the inequalities (\ref{eq:flag}) 
  can be strengthened to 
    \begin{equation} \label{eq:flag2}
      \pphvec_i (\Delta) \ge {d \choose i}
    \end{equation} 
  for $0 \le i \le d-2$. The proof uses results from \cite{NS2} and then follows 
  the general outline of the proof of Corollary \ref{Flag}. 
    
  \section*{Acknowledgments} The authors thank Anders Bj\"orner for suggesting 
    the terminology Buchsbaum* complex and for pointing out useful references. 
    They also thank Isabella Novik and Ed Swartz for their comments on an earlier 
    version of this paper and for their help with the formulation of Question 
    \ref{que:phh}. The first author was supported by the 70/4/8755 ELKE research fund 
    of the University of Athens.

\end{document}